\documentclass[12pt]{amsart}

\usepackage{graphicx}
\usepackage{amssymb}
\usepackage{float}
\usepackage{fullpage}

\newtheorem{theorem}{Theorem}[section]
\newtheorem{lemma}[theorem]{Lemma}
\newtheorem{corollary}[theorem]{Corollary}
\newtheorem{proposition}[theorem]{Proposition}
\newtheorem{conjecture}[theorem]{Conjecture}

\theoremstyle{definition}

\theoremstyle{remark}

\numberwithin{equation}{section}

\newcommand{\PP}{\mathcal{P}}
\newcommand{\C}{{\mathbb C}}
\newcommand{\R}{{\mathbb R}}
\newcommand{\Q}{{\mathbb Q}}
\newcommand{\Z}{{\mathbb Z}}

\newcommand{\li}{\textnormal{li}}

\renewcommand{\Re}{\mathop{\rm Re}}
\renewcommand{\Im}{\mathop{\rm Im}}

\begin{document}

\title[Primes in prime number races]{Primes in prime number races}

\author{Jared Duker Lichtman}
\address{DPMMS, Centre for Mathematical Sciences, University of Cambridge, Wilberforce Road, Cambridge CB3 0WB, UK}

\email{jdl65@cam.ac.uk}

\author{Greg Martin}
\address{Department of Mathematics, University of British Columbia, Room 121, 1984 Mathematics Road, Vancouver, BC, Canada V6T 1Z2}
\email{gerg@math.ubc.ca}

\author{Carl Pomerance}
\address{Department of Mathematics, Dartmouth College, Hanover, NH 03755}
\email{carl.pomerance@dartmouth.edu}


\date{January 3, 2019.}



\begin{abstract}
Rubinstein and Sarnak  have shown,  conditional on the Riemann hypothesis (RH)  and the linear independence hypothesis (LI) on the non-real zeros of $\zeta(s)$, that
the set of real numbers $x\ge2$ for which $\pi(x)>\li(x)$ has a logarithmic density, which they computed to be
about $2.6\times10^{-7}$. 
A natural problem is to examine the actual primes in this race. We prove, assuming RH and LI, that the logarithmic density of the set of primes $p$ for which $\pi(p)>\li(p)$ relative to the prime numbers exists
and is the same as the Rubinstein--Sarnak density.    We also extend such results to a broad class of`prime number races, including the ``Mertens race" between $\prod_{p< x}(1-1/p)^{-1}$ and $e^{\gamma}\log x$ and the ``Zhang race" between $\sum_{p\ge x}1/(p\log p)$ and $1/\log x$. These latter results resolve a question of the first and third author from a previous paper, leading to further progress on a 1988 conjecture of Erd\H os on primitive sets.
\end{abstract}

\maketitle



\section{Introduction}

In the early twentieth century it was noticed that while the prime-counting function $\pi(x)$
and the logarithmic integral function $\li(x) = \int_0^x dt/\log t$ are satisfyingly close together for all values of $x$ where
both had been computed, $\li(x)$ always seemed to be slightly larger than $\pi(x)$.  It was a breakthrough when
Littlewood~\cite{Litt} proved that in fact the sign of $\li(x)-\pi(x)$ changes infinitely often as $x\to\infty$.
We still do not know a specific numerical value of $x$ for which this difference is negative, but the smallest such value
is suspected to be very large, near $1.4\times10^{316}$ (see~\cite{BH} and subsequent refinements).  We do know that $\pi(x)<\li(x)$ for all $2\le x\le10^{19}$,
thanks to calculations of B\"uthe~\cite{Bu}.

Another important development concerning this ``race" between $\pi(x)$ and $\li(x)$ was the
paper of Rubinstein and Sarnak~\cite{RubSarn}.  Assuming some standard conjectures
about the zeros of the Riemann zeta-function, namely the Riemann hypothesis and
a linear independence hypothesis on the zeros of $\zeta(\frac12+it)$, they showed that the logarithmic density $\delta(\Pi)$ of
the set
\begin{equation} \label{Pi defn}
\Pi := \{x\in\R_{\ge1}\colon \pi(x)>\li(x)\}
\end{equation}
exists and is a positive number
\begin{equation} \label{Delta defn}
\delta(\Pi) = \Delta \mathop{\dot=} 2.6\times10^{-7}.
\end{equation}
Here, given a set $\mathcal M \subset \R_{\ge1}$, the logarithmic density of $\mathcal M$ is defined as usual as
\begin{align*}
    \delta(\mathcal M) := \lim_{x\to \infty}\frac{1}{\log x}\int_{t\in \mathcal M\cap [1,x]}\frac{dt}{t},
\end{align*}
provided the limit exists. 

Since $\pi(x)$ counts primes, it is natural to consider the actual primes in the race:
What can be said about the set of primes $p$ for which $\pi(p)>\li(p)$?
We define the discrete logarithmic density of a set $\mathcal M \subset \R_{\ge1}$ relative to the prime numbers as
\begin{align*}
    \delta'(\mathcal M) := \lim_{x\to \infty}\frac{1}{\log\log x}\sum_{\substack{p \le x\\p\in \mathcal M}}\frac{1}{p},
\end{align*}
if the limit exists.
Due to the partial summation formula
\begin{align*}
    \sum_{\substack{p \le x\\p\in \mathcal M}}\frac{1}{p} = \frac{1}{\log x}\sum_{\substack{p \le x\\p\in \mathcal M}}\frac{\log p}{p} + \int_2^x \frac{1}{t\log^2 t}\sum_{\substack{p \le t\\p\in \mathcal M}}\frac{\log p}{p}\,dt,
\end{align*}
we see that if the modified limit
\begin{equation} \label{d^* defn}
    \delta^*(\mathcal M) := \lim_{x\to \infty}\frac{1}{\log x}\sum_{\substack{p \le x\\p\in \mathcal M}}\frac{\log p}{p}
\end{equation}
exists, then it is equal to $\delta'(\mathcal M)$. (The converse does not hold in general, since $\delta^*(\mathcal M)$ might not exist even if $\delta'(\mathcal M)$ does. For example, let $\mathcal P_k$ be the set of all primes between $2^{(2k-1)!}$ and $2^{(2k)!}$, and let $\mathcal P=\bigcup_{k\ge1} \mathcal P_k$. Then $\delta'(\mathcal P) = 1/2$ but $\delta^*(\mathcal P)$ does not exist.)
We shall find it more convenient to deal with $\delta^*(\mathcal M)$ in our proofs below. We also let $\overline{\delta}^*$ and $\underline{\delta}^*$ denote the expression on the right-hand side of equation~\eqref{d^* defn} with $\lim$ replaced by $\limsup$ and $\liminf$, respectively.

Our general philosophy is that the primes are reasonably randomly distributed; in particular, there seems to be no reason for the primes to conspire to lie in the set of real numbers $\Pi$ any more or less often than expected. 
With the aid of an old theorem of Selberg that most short intervals contain the ``right" number of primes,
we prove that there is no such conspiracy; more precisely we prove, under the same two assumptions as Rubinstein and Sarnak, that $\delta^*(\Pi)=\Delta$ (see Theorem~\ref{thm:pivsli}). Moreover, we prove similar results---comparing the logarithmic density of a set of real numbers to the relative logarithmic density of the primes lying in that set---for a number of other prime races, some of which have not been considered before (see Theorems~\ref{thm:mertens} and~\ref{Zhang primes density thm}).  These results resolve some problems from \cite{LP} and so make progress
on the Erd\H os conjecture on primitive sets (see Section \ref{sec:mertens}).

Finally we remark that our approach applies equally to prime races involving residue classes.
To do this one would replace Selberg's theorem on the distribution of primes in almost all short intervals 
with a result of Koukoulopoulos \cite[Theorem 1.1]{K} which does the same for
primes in a residue class to a fixed modulus.

\section{A key result}

For a ``naturally occurring'' set $\mathcal M$ of real numbers for which $\delta(\mathcal M)$ exists,
it is natural to wonder how $\delta^*(\mathcal M)$ compares to $\delta(\mathcal M)$. 
We prove the two densities are equal in the case of sets of the form
\begin{equation} \label{M_a defn}
\mathcal M_a(f) = \{ x\colon f(x) > a\}
\end{equation}
for functions $f$ that are suitably nice.

\begin{theorem} \label{thm:density}
Consider a function $f\colon\R_{\ge1}\to \R$ satisfying the following two conditions:
\begin{enumerate}
    \item[(a)] For all real numbers $a>b$, there exists $x_0=x_0(a,b)$ such that for all $x\ge x_0$, if $f(x)>a$ then $f(z)>b$ for all $z\in [x,x+x^{1/3}]$; and similarly for the function $-f(x)$.
    \item[(b)]  The function $f$ has a continuous logarithmic distribution function: for all $a\in\R$, the set $\mathcal M_a(f)$ has a well-defined logarithmic density $\delta(\mathcal M_a(f))$, and the map $a\mapsto \delta(\mathcal M_a(f))$ is continuous.
\end{enumerate}
Then for every real number $a$, the relative density $\delta^*(\mathcal M_a(f))$ exists and is equal to $\delta(\mathcal M_a(f))$.
\end{theorem}

\noindent 
It is worth noting that the assumptions and conclusion of the theorem imply that the relative density map $a\mapsto \delta^*(\mathcal M_a(f))$ is also continuous; in particular, ``ties have density~$0$'', meaning that $\delta^*(\{x\colon f(x)=a\}) = 0$. Thus there is no difference between considering $f(x)>a$ and considering $f(x)\ge a$ in the situations we investigate.

Recall the Linear Independence hypothesis (LI), which asserts that the sequence of numbers $\gamma_n > 0$ such that $\zeta(\tfrac{1}{2} + i\gamma_n) = 0$ is linearly independent over $\Q$.

\begin{theorem}\label{thm:pivsli}
Let the set $\Pi$ and the number $\Delta=\delta(\Pi)$ be defined as in equations~\eqref{Pi defn} and~\eqref{Delta defn}, respectively. Assuming RH and LI, the discrete logarithmic density of $\Pi$ relative to the primes is $\delta^*(\Pi) = \Delta$.
\end{theorem}

\begin{proof}[Proof of Theorem \ref{thm:pivsli} via Theorem \ref{thm:density}]
Consider the normalized error function
$$
E_{\pi}(x)=\frac{\log x}{\sqrt{x}}(\pi(x) - \li(x)),
$$
and note that $\Pi = \{x\colon \pi(x) > \li(x)\} = \mathcal M_0(E_{\pi})$. It thus suffices to show that $E_{\pi}$ satisfies conditions~(a) and~(b) of Theorem~\ref{thm:density}.

Consider any number $z\in[x,x+x^{1/3}]$. We have $\pi(z)-\pi(x) \le x^{1/3}$ and $\li(z) - \li(x) \le x^{1/3}$ trivially, and hence $|E_{\pi}(z)-E_{\pi}(x)| \le (2\log x)/{x^{1/6}}$.  Since the right-hand side tends to $0$, this inequality easily implies condition~(a) of the theorem.

Moreover, condition (b)---namely the fact that $E_{\pi}$ has a continuous limiting logarithmic distribution---is a consequence (under RH and LI) of the work of Rubinstein and Sarnak: first, they establish a formula for the Fourier transform of this limiting logarithmic distribution (see~\cite[equation~(3.4) and the paragraph following]{RubSarn}). They then argue that this Fourier transform is rapidly decaying (see~\cite[Section 2.3]{FM} for a more explicit version of their method). From this they conclude that the distribution itself is continuous (and indeed much more, namely that it corresponds to an analytic density function---see~\cite[Remark 1.3]{RubSarn}).
\end{proof}

The continuity of the limiting logarithmic distribution of $E_{\pi}$ can be deduced from a substantially weakened version of LI: indeed, we only require the imaginary part of one nontrivial zero of $\zeta(s)$ to not be a rational linear combination of other such imaginary parts (see~\cite[Theorem 2.2(2)]{Devin}).

In the next section we prove Theorem \ref{thm:density}, which will complete the proof of Theorem~\ref{thm:pivsli}.

\section{Proof of Theorem \ref{thm:density}}
We begin with some notation. For any interval $I$ of real numbers, let $\pi(I)$ denote the number of primes in~$I$. For any positive real number $y$, define the half-open interval $I(y):=(y,y+y^{1/3}]$. Define an increasing sequence of real numbers recursively by $y_1=1$ and $y_{k+1} = y_k+y_k^{1/3}$ for $k\ge1$, and let $I_k:=I(y_k) = (y_k,y_{k+1}]$. We have thus partitioned $\R_{>1} = \bigcup_{k=1}^\infty I_k$ into a disjoint union of short half-open intervals.

\begin{lemma} \label{yk sum lemma}
For any fixed real number $\alpha>\frac23$, we have $\sum_{k=1}^\infty y_k^{-\alpha} \ll 1$.
\end{lemma}

\begin{proof}
For any $U\ge1$, the number of integers $k$ such that $y_k\in[U,2U)$ is at most $U^{2/3}$, since the length of each corresponding interval $I_k$ is at least $U^{1/3}$. Therefore
\[
\sum_{k=1}^\infty y_k^{-\alpha} = \sum_{j=0}^\infty \sum_{k\colon y_k\in[2^j,2^{j+1})} y_k^{-\alpha} \le \sum_{j=0}^\infty (2^j)^{2/3} (2^j)^{-\alpha} = \frac1{1-2^{2/3-\alpha}},
\]
since we assumed $\alpha>\frac23$.
\end{proof}

Given $\epsilon>0$, we say that an interval $I(y)$ is $\epsilon$-good if
\[
     \bigg|\pi(I(y)) - \frac{y^{1/3}}{\log y} \bigg| \le \frac{\epsilon y^{1/3}}{\log y},
\]
and otherwise we say that $I(y)$ is $\epsilon$-bad. Selberg~\cite{Selb} showed
 that there exists a set $\mathcal S\subset \R_{\ge1}$ whose natural density equals $1$ for which
\begin{align}
\label{eq:sel}
    \pi(y+y^{1/3}) - \pi(y) \sim \frac{y^{1/3}}{\log y} \quad \text{for all } y\in \mathcal S.
\end{align}
This implies that for any $\epsilon>0$, the set of real numbers $y$ for which $I(y)$ is $\epsilon$-bad has density~$0$.  
(Selberg \cite{Selb} proved \eqref{eq:sel} where the exponent ``$1/3$" is permitted to be any constant in $(19/77,1]$.
Selberg's
theorem has been subsequently improved: from Huxley \cite{H}, one may take the exponent in \eqref{eq:sel}
as any number in $(1/6,1]$, 
see \cite[(1.3)]{K}.)

Our next lemma shows that $\epsilon$-bad intervals among the $I_k$ are also sparse.

\begin{lemma}\label{lemma:bad}
For each $\epsilon > 0$, the union of the $\epsilon$-bad intervals $I_k$ has natural density~$0$, and hence logarithmic density~$0$.
\end{lemma}

\begin{proof}
For every $k\ge1$, define $J_k := (y_k, y_k + \frac\epsilon{14} y_k^{1/3}]$.
Suppose that $k\ge1$ is chosen so that $I_k$ is an $\epsilon$-bad interval. Note that for all $y \in J_k$, the intervals $I(y)$ and $I_k = I(y_k)$ have nearly the same number of primes; more precisely,
\begin{equation}  \label{lil pi difference}
\pi(I(y)) - \pi(I_k) = \pi\big((y_{k+1},y + y^{1/3}]\big) - \pi\big((y_k,y]\big),
\end{equation}
since the primes in the larger interval $(y,y_{k+1}]$ cancel in the difference.
By Titchmarsh's inequality~\cite[equation~(1.12)]{MV}, we have $\pi(I) \le 2h/\log h$ for all intervals $I$ of length $h>1$;
and since $2h/\log h$ is an increasing function of $h$ for $h>e$, we deduce that for any interval $I$ of length at most $\tfrac\epsilon{13} y_k^{1/3}$,
\[
\pi(I) \le \frac{2\cdot \tfrac\epsilon{13} y_k^{1/3}}{\log(\tfrac\epsilon{13} y_k^{1/3})} < \frac\epsilon2 \frac{y_k^{1/3}}{\log y_k}
\]
when $k$ is sufficiently large in terms of~$\epsilon$. (This deduction assumed that the length of $I$ exceeds $e$, but the final inequality is trivial for large $k$ when the length of $I$ is at most~$e$.)
In particular, both intervals on the right-hand side of equation~\eqref{lil pi difference} have length at most $\tfrac\epsilon{13} y_k^{1/3}$ when $k$ is sufficiently large, from which we see that $\big| \pi(I(y)) - \pi(I_k) \big| \le \frac\epsilon2 {y_k^{1/3}}/\log y_k$. Consequently, since $I_k$ is $\epsilon$-bad,
we conclude that
\begin{align*}
     \bigg|\pi(I(y)) - \frac{y^{1/3}}{\log y} \bigg| &\ge \bigg|\pi(I_k) - \frac{y_k^{1/3}}{\log y_k} \bigg| - \big|\pi(I(y)) - \pi(I_k) \big| - \bigg| \frac{y^{1/3}}{\log y} - \frac{y_k^{1/3}}{\log y_k} \bigg| \\
     &\ge \frac{\epsilon y_k^{1/3}}{\log y_k} - \frac\epsilon2 \frac{y_k^{1/3}}{\log y_k} + o(1) > \frac\epsilon3 \frac{y^{1/3}}{\log y}
\end{align*}
when $k$ is sufficiently large (where the mean value theorem was used in the middle inequality).
In other words, we have shown that $I_k$ being $\epsilon$-bad implies that $I(y)$ is $\frac\epsilon3$-bad for all $y\in J_k$.

Let $J$ be the (disjoint) union of all the intervals $J_k$, where $k$ ranges over those positive integers for which $I_k$ is $\epsilon$-bad.
By the result of Selberg described above, the set of $\frac\epsilon3$-bad real numbers (which contains $J$) has density~$0$, so $J\cap[1,x]$ has measure~$o(x)$. But this measure is at least $\frac\epsilon{14}$ times the measure of the union of all $\epsilon$-bad intervals~$I_k$; hence, the union of these intervals below $x$ also has measure~$o(x)$, which completes the proof.
\end{proof}

\begin{proof}[Proof of Theorem~\ref{thm:density}]
For the sake of simplicity, we abbreviate $\mathcal M_a(f)$ to $\mathcal M_a$ during this proof.
Let $\epsilon$ and $\eta$ be positive parameters, and let $\mathcal B_\epsilon$ denote the union of all $\epsilon$-bad intervals of the form $I_k$, so that $\mathcal B_\epsilon$ has logarithmic density $0$ by Lemma~\ref{lemma:bad}.

Suppose that $I(y)$ is any $\epsilon$-good interval. Since $\int_{I(y)} dt/t = \int_y^{y+y^{1/3}} dt/t = \log(1+y^{-2/3}) = y^{-2/3}+O(y^{-4/3})$, we see that
\begin{equation} \label{good loggy bound}
    \sum_{p\in I(y)}\frac{\log p}{p} \le \frac{\log y}{y}\pi(I(y)) \le (1+\epsilon)y^{-2/3} = (1+\epsilon)\int_{I(y)} \frac{dt}{t} + O(y^{-4/3}),
\end{equation}
where the second inequality used the $\epsilon$-goodness of~$I(y)$. On the other hand, even if $I(y)$ is an $\epsilon$-bad interval, Titchmarsh's inequality still yields
\begin{equation} \label{bad loggy bound}
    \sum_{p\in I(y)}\frac{\log p}{p} \le \frac{\log y}{y}\pi(I(y)) \ll \frac{\log y}{y} \frac{y^{1/3}}{\log(y^{1/3})} \ll y^{-2/3} \ll \int_{I(y)} \frac{dt}{t}.
\end{equation}

By condition (a) of Theorem \ref{thm:density}, there exists a positive integer $C$ (depending on $a$ and~$\eta$) such that if $p$ is a prime in an interval $I_k$ with $k>C$, then the inequality $f(p)>a$ implies that $f(z) > a-\eta$ for all $z\in I_k$. In particular, every $I_k$ containing a prime $p$ with $f(p)>a$ is either a subset of $\mathcal B_\epsilon$ or else is an $\epsilon$-good interval contained in $\mathcal M_{a-\eta}$, so that
\begin{align*}
    \sum_{\substack{p\le x\\f(p)>a}}\frac{\log p}{p}\le \sum_{\substack{I_k\subset \mathcal M_{a-\eta}\cap[1,x] \\ I_k \text{ is $\epsilon$-good}}}\sum_{p\in I_k}\frac{\log p}{p} + \sum_{I_k\subset \mathcal B_\epsilon\cap[1,x]}\sum_{p\in I_k}\frac{\log p}{p}.
\end{align*}
Using equation~\eqref{good loggy bound} for the terms in the first sum and equation~\eqref{bad loggy bound} for the second sum, we obtain the upper bound
\begin{align*}
    \sum_{\substack{p\le x\\f(p)>a}}\frac{\log p}{p} &\le (1+\epsilon)\int_{\mathcal M_{a-\eta}\cap[1,x]} \frac{dt}{t} + O\bigg(\sum_{y_k\le x}y_k^{-4/3}\bigg) + O\bigg( \int_{\mathcal B_\epsilon\cap[1,x]} \frac{dt}{t} \bigg) \\
    & \le (1+\epsilon)\int_{\mathcal M_{a-\eta}\cap[1,x]} \frac{dt}{t} + O(1) + o(\log x)
\end{align*}
by Lemma~\ref{yk sum lemma} and the fact that $\mathcal B_\epsilon$ has logarithmic density~$0$.

Therefore we have
\begin{align}\label{eq:over}
 \overline{\delta}^*(\mathcal M_a) &= \limsup_{x\to\infty}\frac{1}{\log x}\sum_{\substack{p\le x\\f(p)>a}}\frac{\log p}{p} \\
 &\le \limsup_{x\to\infty} \bigg( \frac{(1+\epsilon)}{\log x}\int_{\mathcal M_{a-\eta}\cap[1,x]} \frac{dt}{t} + o(1) \bigg) = (1+\epsilon)\delta(\mathcal M_{a-\eta})
\end{align}
since $\delta(\mathcal M_{a-\eta})$ exists by condition~(b) of Theorem \ref{thm:density}.

Similarly, the primes in $\mathcal M_a$ that are contained in $\epsilon$-good intervals $I_k\subset \mathcal M_a$ form a subset of all primes in $\mathcal M_a$. Then for a lower bound, it suffices to consider the $\epsilon$-good intervals in $\mathcal M_a$, which by a simple computation gives the bound $\underline{\delta}^*(\mathcal M_a) \ge (1-\epsilon)\delta(\mathcal M_a)$.

Since these bounds hold for all $\epsilon>0$, we see that
\begin{align}
    \delta(\mathcal M_a) \le \underline{\delta}^*(\mathcal M_a) \le \overline{\delta}^*(\mathcal M_a) \le \delta(\mathcal M_{a-\eta}).
\end{align}
Finally, by condition (b) the map $\eta\mapsto \delta(\mathcal M_{a-\eta})$ is continuous, so since $\eta>0$ was arbitrary we conclude that $\underline{\delta}^*(\mathcal M_a) = \overline{\delta}^*(\mathcal M_a)= \delta(\mathcal M_a)$ as desired.
\end{proof}

\section{The Mertens race}
\label{sec:mertens}

In 1874, Mertens proved three remarkable and related
 results on the distribution of prime numbers. His third theorem asserts that  
\begin{align*}
    \prod_{p < x}\Big(1 - \frac{1}{p}\Big)^{-1} \sim e^\gamma \log x \quad\textnormal{as }x\to\infty,
\end{align*}
where $\gamma$ is the Euler--Mascheroni constant. The ``Mertens race'' between $e^\gamma \log x$ and this product of Mertens is mathematically analagous to the race between $\li(x)$ and $\pi(x)$. Recent analysis of Lamzouri \cite{Lamz} implies, conditionally on RH and LI, that the normalized error function
\begin{align*}
    E_M(x) = \bigg( \log \prod_{p < x}\Big(1 - \frac{1}{p}\Big)^{-1} - \log\log x -\gamma \bigg) \sqrt{x}\log x
\end{align*}
possesses the exact same limiting distribution as that of
\begin{align*}
    -E_{\pi}(x) = \frac{\log x}{\sqrt{x}}\Big(\li(x) - \pi(x) \Big)
\end{align*}
that appeared in the proof of Theorem~\ref{thm:pivsli}.
 We say a prime $p$ is \textbf{Mertens} if $E_M(p) > 0$. It can be checked that the first $10^8$ odd primes are Mertens. The first and third authors have shown~\cite[Theorem 1.3]{LP}, assuming RH and LI, that the lower relative logarithmic density of the Mertens primes exceeds $.995$.  Applying Theorem \ref{thm:density} to the Mertens race by choosing $f(x) = E_M(x)$ leads immediately to the following improvement.

\begin{theorem}\label{thm:mertens}
Assuming RH and LI, the Mertens primes have relative logarithmic density $1-\Delta$, where $\Delta$ was defined in equation~\eqref{Delta defn}.
\end{theorem}
\begin{proof}
Consider any number $z\in[x,x+x^{1/3}]$. We have $\log\log z - \log \log x \ll 1/(x\log x)$ by the mean value theorem and 
\begin{align*}
    \sum_{x \le p < z}\log\Big(1 - \frac{1}{p}\Big)^{-1} \ll \sum_{x \le p < z} \frac1p \le \sum_{x\le n< x+x^{1/3}} \frac1n < x^{-2/3},
\end{align*}
and so $|E_M(z)-E_M(x)| \ll (\log x)/x^{1/6}$. The fact that this upper bound tends to $0$ easily implies condition~(a) of Theorem \ref{thm:density}. Finally, condition (b) is satisfied by a similar argument as in the proof of Theorem~\ref{thm:pivsli}, using work of Lamzouri~\cite{Lamz} on the limiting logarithmic distribution of $E_M(x)$.
\end{proof}

Theorem \ref{thm:mertens} has an interesting application to the Erd\H os conjecture for primitive sets.
A subset of the integers larger than $1$ is {\bf primitive} if no member divides another. Erd\H os \cite{E35} proved in 1935 that the sum of $1/(a\log a)$ for $a$ running over a primitive set $A$
is universally bounded over all choices for $A$. Some years later in a 1988 seminar in Limoges, he asked if this universal bound is
attained for the set of prime numbers.
If we define $f(a)=1/(a\log a)$ and $f(A)=\sum_{a\in A}f(a)$, and let
$\PP(A)$ denote the set of primes that divide some member of~$A$, then this conjecture is seen
to be equivalent to the following assertion.

\begin{conjecture}[Erd\H os]\label{conj:Erdos}
For any primitive set $A$, we have
$f(A)  \le f(\PP(A))$.
\end{conjecture}

The Erd\H{o}s conjecture remains open, but progress has been made in certain cases. Say a
prime $p$ is {\bf Erd\H os-strong} if $f(A)\le f(p)$ for any primitive set $A$ such that each
member of $A$ has $p$ as its least prime factor.  By partitioning the elements of $A$ into sets $A'_p$ by their smallest prime factor $p$, it is clear that the Erd\H{o}s conjecture would follow if every
prime $p$ is Erd\H os-strong.    The first and third authors~\cite[Corollary 3.0.1]{LP}
proved that every Mertens prime is Erd\H{o}s-strong. 
In particular, the Erd\H{o}s conjecture holds for any primitive set $A$ such that, for all $a\in A$, the smallest prime factor of $a$ is Mertens.

In~\cite{LP} it was conjectured that all primes are Erd\H{o}s-strong. Since $2$ is not a Mertens prime, it would be great progress just to be able to prove that $2$ is Erd\H{o}s-strong. Nevertheless, Theorem \ref{thm:mertens} implies the following corollary.

\begin{corollary}
Assuming RH and LI, the lower relative logarithmic density of the Erd\H{o}s-strong primes is at least $1-\Delta$. In particular, the Erd\H os conjecture holds for all primitive sets whose elements have smallest prime factors in a set of primes of lower relative logarithmic density at least $1-\Delta$.
\end{corollary}

\section{The Zhang race}


By the prime number theorem, one has the asymptotic relation
\begin{align*}
 \sum_{p\ge x}\frac{1}{p\log p} \sim     \frac{1}{\log x} 
\end{align*}
as $x\to\infty$,
and by inspection one further has
\begin{align}\label{eq:Zhang}
 \sum_{p\ge x}\frac{1}{p\log p}\le    \frac{1}{\log x}
\end{align}
for a large range of $x$. 
Beyond its aesthetic appeal, this inequality arises quite naturally in the study of primitive sets. Indeed, Z.~Zhang~\cite{Z} used a weakened version of \eqref{eq:Zhang} to prove Conjecture~\ref{conj:Erdos} for all primitive sets whose elements have at most 4 prime factors, which represented the first significant progress in the literature after \cite{E35}.

Call a prime $q$ {\bf Zhang} if the inequality~\eqref{eq:Zhang} holds for $x=q$. From computations in \cite{LP}, the first $10^8$ primes are all Zhang except for $q=2,3$. Following some ideas of earlier work of Erd\H os and Zhang \cite{EZ}, the first and third authors have shown~\cite[Theorem 5.1]{LP} that Conjecture~\ref{conj:Erdos}
holds for any primitive set
$A$ such that every member of $\PP(A)$ is Zhang.

 We wish to find the density of $\mathcal N$, the set of real numbers for which the Zhang inequality~\eqref{eq:Zhang} holds. Note that $x\in\mathcal N$ if and only if the normalized error
\begin{align} \label{EZ defn}
    E_Z(x) := \bigg(\frac{1}{\log x} - \sum_{p\ge x}\frac{1}{p\log p}\bigg)\sqrt{x}\log^2 x
\end{align}
is nonnegative. To show the density of $\mathcal N$ exists we follow the general plan laid out by  Lamzouri \cite{Lamz}, who proved analogous results for the Mertens race, with some important modifications.

\subsection{Explicit formula for $E_Z(x)$}

First we relate the sum over primes, $\sum_{p\ge x} 1/(p\log p)$, to the corresponding series over prime powers, $\sum_{n\ge x} \Lambda(n)/(n\log^2 n)$, in the following lemma.

\begin{lemma} \label{lem:Ep}
For all $x>1$,
\[
    E_Z(x) = \bigg(\frac{1}{\log x} - \sum_{n\ge x}\frac{\Lambda(n)}{n\log^2 n}\bigg)\sqrt{x}\log^2 x + 1 + O\Big(\frac{1}{\log x}\Big).
\]
\end{lemma}

\begin{proof}
Our first step is to convert the sum over primes to prime powers, via
\begin{align} \label{here's Lambda}
    \sum_{p\ge x}\frac{1}{p\log p} & = \sum_{n\ge x}\frac{\Lambda(n)}{n\log^2 n} - \sum_{\substack{p^k> x\\ k\ge2}}\frac{1}{k^2 p^k \log p}.
\end{align}
The prime number theorem gives that $\pi(y) = y/\log y + O(y/\log^2 y)$, so for any $y\ge2$,
\begin{align*}
\sum_{p\ge y}\frac{1}{p^2\log p} &
= -\frac{\pi(y)}{y^2\log y} + \int_{y}^\infty\frac{2\log t + 1}{t^3\log^2 t}\pi(t)\;dt\\
& = -\frac{1}{y\log^2 y} + O\bigg(\frac{1}{y\log^3 y} \bigg) + \int_{y}^\infty\frac{2\log t + 1}{t^2\log^3 t} \bigg( 1 + O\bigg(\frac1{\log t} \bigg)\bigg) \;dt \\
& = -\frac{1}{y\log^2 y} + \frac{2}{y\log^2 y} + O\bigg(\frac{1}{y\log^3 y}\bigg) = \frac{1}{y\log^2 y} + O\bigg(\frac{1}{y\log^3 y}\bigg).
\end{align*}
In particular, taking $y=\sqrt x$,
\begin{align} \label{squares}
    \sum_{p^2>x}\frac{1}{4p^2\log p} & = \frac{1}{\sqrt{x}\log^2 x} + O\bigg(\frac{1}{\sqrt{x}\log^3 x}\bigg).
\end{align}
For the larger powers of primes, we have
\begin{align*}
    \sum_{p^k> x}\frac{1}{p^k\log p} < \sum_{n> x^{1/k}} \frac1{n^k} < \frac1{\lceil x^{1/k} \rceil^k} + \int_{x^{1/k}}^\infty \frac{dt}{t^k} \le \frac1x + \frac1{(k-1)x^{1-1/k}} \ll x^{-2/3}
\end{align*}
uniformly for $k\ge3$, and thus
\begin{align} \label{higher powers}
    \sum_{k\ge3}\sum_{p^k> x}\frac{1}{k^2p^k\log p} \ll \sum_{k\ge3}\frac{x^{-2/3}}{k^2} \ll x^{-2/3}.
\end{align}
Inserting the estimates~\eqref{squares} and~\eqref{higher powers} into equation~\eqref{here's Lambda} then yields
\begin{align*}
     \sum_{p\ge x}\frac{1}{p\log p} & = \sum_{n\ge x}\frac{\Lambda(n)}{n\log^2 n} - \frac{1}{\sqrt{x}\log^2 x} + O\bigg(\frac{1}{\sqrt{x}\log^3 x}\bigg),
\end{align*}
which implies the statement of the lemma.
\end{proof}

By integrating twice, we relate our series $\sum \Lambda(n)/n\log^2 n$ to the series $\sum \Lambda(n)/n^a = -\zeta'/\zeta(a)$, which is more amenable to contour integration. This leads to the following explicit formula for $E_Z(x)$ over the zeros of $\zeta(s)$, analogous to \cite[Proposition 2.1]{Lamz}.

\begin{proposition}\label{prop:Ex}
Unconditionally, for any real numbers $x,T\ge 5$,
\begin{align*}
    E_Z(x) & = 1 - \sum_{|\Im(\rho)|< T} \frac{x^{\rho-1/2}}{\rho-1} + O\bigg(\frac{1}{\log x} + \frac{\sqrt{x}}{T}\log^2(xT) + \frac{1}{\log x}\sum_{|\Im(\rho)|< T}\frac{x^{\Re(\rho)-1/2}}{\Im(\rho)^2}\bigg),
\end{align*}
where $\rho$ runs over the nontrivial zeros of $\zeta(s)$.
\end{proposition}

\begin{proof}
Our starting point is a tool from Lamzouri, namely \cite[Lemma 2.4]{Lamz}: for any real numbers $a>1$ and $x,T\ge5$,
\begin{align*}
    \sum_{n< x}\frac{\Lambda(n)}{n^a} & = -\frac{\zeta'}{\zeta}(a) + \frac{x^{1-a}}{1-a} - \sum_{|\Im(\rho)|< T}\frac{x^{\rho-a}}{\rho-a}\\
    & \qquad + O\bigg(x^{-a}\log x + \frac{x^{1-a}}{T}\Big(4^a + \log^2 x + \frac{\log^2 T}{\log x}\Big) + \frac{1}{T}\sum_n\frac{\Lambda(n)}{n^{a+1/\log x}} \bigg).
\end{align*}
Then integration with respect to $a$ gives for any $b>1$,
\begin{align*}
    \sum_{n< x}\frac{\Lambda(n)}{n^b\log n} & = \int_b^\infty \sum_{n< x}\frac{\Lambda(n)}{n^a}\;da\\
    & = \log\zeta(b) + \int_b^\infty \frac{x^{1-a}}{1-a}\;da - \sum_{|\Im(\rho)|< T}\int_b^\infty \frac{x^{\rho-a}}{\rho-a}da + E_1,
\end{align*}
where
\begin{align*}
    E_1 \ll x^{-b} + \frac{x^{1-b}}{T}\Big(\frac{4^b}{\log x} + \log x + \frac{\log^2 T}{\log^2 x}\Big) + \frac{1}{T}\sum_n\frac{\Lambda(n)}{n^{b+1/\log x}\log n}.
\end{align*}
Integrating once again with respect to $b$, we have
\begin{align}\label{eq:doubleint}
    \sum_{n< x}\frac{\Lambda(n)}{n\log^2 n}  &= \int_1^\infty \sum_{n< x}\frac{\Lambda(n)}{n^b\log n}\;db\nonumber\\
    &~ = \int_1^\infty \log\zeta(b)\;db + \int_1^\infty\int_b^\infty \frac{x^{1-a}}{1-a}\;da\;db - \sum_{|\Im(\rho)|< T} \int_1^\infty\int_b^\infty \frac{x^{\rho-a}}{\rho-a}\;da\;db + E_2,
\end{align}
where
\begin{align*}
    E_2 & \ll \frac{1}{x\log x} + \frac{1}{T}\Big(\frac{4}{\log^2 x} + 1 + \frac{\log^2 T}{\log^3 x}\Big) + \frac{1}{T}\sum_n\frac{\Lambda(n)}{n^{1+1/\log x}\log^2 n} \\
    & \ll \frac{1}{x\log x} + \frac{1}{T}\Big(1 + \frac{\log^2 T}{\log^3 x}\Big) + \frac{1}{T}\sum_n\frac{\Lambda(n)}{n\log^2 n} \ll \frac{1}{x\log x} + \frac{1}{T}\Big(1 + \frac{\log^2 T}{\log^3 x}\Big),
\end{align*}
since $\sum_n \Lambda(n)/(n\log^2n)\ll1$.

The first term on the right-hand side of equation~\eqref{eq:doubleint} can be written as
\begin{align} \label{first term}
    \int_1^\infty \log\zeta(b)\;db = \int_1^\infty \sum_n\frac{\Lambda(n)}{n^b\log n}\;db = \sum_n \frac{\Lambda(n)}{n\log^2 n}, 
\end{align}
where the Fubini--Tonelli theorem justifies the interchange of summation and integration
since all terms are nonnegative.
The second term on the right-hand side of equation~\eqref{eq:doubleint} evaluates to
\begin{align}
    \int_1^\infty\int_b^\infty \frac{x^{1-a}}{1-a}\;da\;db &= \int_1^\infty \frac{x^{1-a}}{1-a} \bigg(  \int_1^a \;db \bigg) \;da  \notag \\
    &= -\int_1^\infty x^{1-a}\;da = \frac{x^{1-a}}{\log x}\bigg|_1^\infty =-\frac{1}{\log x},  \label{second term}
\end{align}
where the interchange of integrals is again justified by the Fubini--Tonelli theorem.

The double integral inside the series on the right-hand side of equation~\eqref{eq:doubleint} is evaluated using a similar calculation:
\begin{align*}
    -\int_1^\infty\int_b^\infty \frac{x^{\rho-a}}{\rho-a}\;da\;db & = -\int_1^\infty\frac{a-1}{\rho-a}x^{\rho-a}\;da\\
    & = \int_1^\infty x^{\rho-a}\;da - (\rho-1)\int_1^\infty\frac{x^{\rho-a}}{\rho-a}\;da.
\end{align*}
The first integral comes out to $x^{\rho-1}/\log x$, while for the second, integrating by parts twice gives
\begin{align*}
    (\rho-1)\int_1^\infty\frac{x^{\rho-a}}{\rho-a}\;da = \frac{x^{\rho-1}}{\log x} + \frac{x^{\rho-1}}{(\rho-1)\log^2 x} + \frac{2(\rho-1)}{\log^2 x}\int_1^\infty\frac{x^{\rho-a}}{(\rho-a)^3}\;da.
\end{align*}
Letting $u = (a-1)\log x$, we have $a=1+u/\log x$ so the latter integral becomes
\begin{align*}
    \frac{2(\rho-1)}{\log^2 x}\int_1^\infty\frac{x^{\rho-a}}{(\rho-a)^3}\;da =\frac{2 (\rho-1)x^{\rho - 1}}{\log^3x}\int_0^\infty \frac{e^{-u}}{(\rho -1- u/\log x)^3}\;du.
\end{align*}
Note that $|\rho-1 - u/\log x| \ge |\Im(\rho)|$ for all $u\in\R$, so we deduce
\begin{align*}
    \bigg|\frac{2(\rho-1)}{\log^2 x}\int_1^\infty\frac{x^{\rho-a}}{(\rho-a)^3}\;da\bigg| \ll \frac{x^{\Re(\rho)-1}}{\Im(\rho)^2\log^3 x}.
\end{align*}
Thus we have
\begin{align}
    -\int_1^\infty\int_b^\infty \frac{x^{\rho-a}}{\rho-a}da\;db & = \frac{x^{\rho-1}}{\log x} - \bigg(\frac{x^{\rho-1}}{\log x} + \frac{x^{\rho-1}}{(\rho-1)\log^2 x} + O\bigg(\frac{x^{\Re(\rho)-1}}{\Im(\rho)^2\log^3 x}\bigg) \bigg) \nonumber\\
    & = -\frac{x^{\rho-1}}{(\rho-1)\log^2 x} +  O\bigg(\frac{x^{\Re(\rho)-1}}{\Im(\rho)^2\log^3 x}\bigg). \label{third term}
\end{align}

The calculations~\eqref{first term}, \eqref{second term}, and~\eqref{third term} transform equation~\eqref{eq:doubleint} into
\begin{align*}
    \sum_{n< x}\frac{\Lambda(n)}{n\log^2 n} & = \sum_n \frac{\Lambda(n)}{n\log^2 n} -\frac{1}{\log x} - \frac{1}{\log^2 x}\sum_{|\Im(\rho)|< T}\frac{x^{\rho-1}}{\rho-1}\\
    &\qquad\qquad + O\bigg(\frac{1}{x\log x} + \frac1T\bigg(1+\frac{\log^2 T}{\log^3 x}\bigg)+ \frac{1}{\log^3 x}\sum_{|\Im(\rho)|< T}\frac{x^{\Re(\rho)-1}}{\Im(\rho)^2}\bigg)
\end{align*}
and thus
\begin{align*}
  \frac1{\log x}-  \sum_{n\ge  x}\frac{\Lambda(n)}{n\log^2 n} & = - \frac{1}{\log^2 x}\sum_{|\Im(\rho)|< T}\frac{x^{\rho-1}}{\rho-1}\\
    &\quad\qquad + O\bigg(\frac{1}{x\log x} + \frac{1+\log^2 T/\log^3 x}{T} + \frac{1}{\log^3 x}\sum_{|\Im(\rho)|< T}\frac{x^{\Re(\rho)-1}}{\Im(\rho)^2}\bigg).
\end{align*}
The proposition now follows upon comparing this formula to Lemma~\ref{lem:Ep}.
\end{proof}

If we assume the Riemann hypothesis we obtain the following corollary, analogous to \cite[Corollary 2.2]{Lamz}.

\begin{corollary}\label{cor:Ex}
Assume RH, and let $\tfrac{1}{2} + i\gamma$ run over the nontrivial zeros of $\zeta(s)$ with $\gamma>0$. Then, for any real numbers $x, T \ge 5$ we have
\begin{align}
    E_Z(x) & = 1 - 2\Re\sum_{0<\gamma< T} \frac{x^{i\gamma}}{-1/2+i\gamma} + O\bigg(\frac{1}{\log x} + \frac{\sqrt{x}}{T}\log^2(xT)\bigg),
\end{align}
\end{corollary}

\begin{proof}
By the Riemann--von Mangoldt formula, 
\begin{align*}
    \bigg|\sum_{|\gamma|< T}\frac{x^{i\gamma}}{\gamma^2}\bigg| \le \sum_{|\gamma|<T}\frac{1}{\gamma^2} \ll 1,
\end{align*}
so the corollary now follows from Proposition~\ref{prop:Ex}.
\end{proof}

\subsection{Density Results}


Since the explicit formula for the Zhang primes in Corollary \ref{cor:Ex} is exactly the same as that of the Mertens primes given by Lamzouri (upon noting a typo in~\cite[Corollary 2.2]{Lamz}, namely,  that ``$E_M(x) = 1+ \cdots$'' should read ``$E_M(x) = 1 - \cdots$''), the analysis therein leads to the following results. Recall that $\mathcal N$ is the set of real numbers for which the Zhang inequality~\eqref{eq:Zhang} holds, and that $E_Z(x)$ is defined in equation~\eqref{EZ defn}.

\begin{theorem}\label{thm:E_Z}
Assume RH. Then
\begin{align*}
0 < \underline \delta(\mathcal N) \le \overline \delta(\mathcal N) < 1.
\end{align*}
Moreover, $E_Z(x)$ possesses a limiting distribution $\mu_N$, that is,
\begin{align*}
    \lim_{x\to\infty}\frac{1}{\log x}\int_2^x f(E_Z(t))\;dt = \int_\R f(t)\; d\mu_N(t)
\end{align*}
for all bounded continuous functions $f$ on $\R$.
\end{theorem}

\begin{proposition}
Assume RH and LI. Let $X(\gamma)$ be a sequence of independent random variables, indexed by the positive imaginary parts of the non-trivial zeros of $\zeta(s)$, each of which is uniformly distributed on the unit circle. Then $\mu_N$ is the distribution of the random variable
\begin{align}
    Y = 1 - 2\Re \sum_{\gamma > 0}\frac{X(\gamma)}{\sqrt{1/4+\gamma^2}}.
\end{align}
\end{proposition}

\begin{theorem} \label{Zhang primes density thm}
Assume RH and LI. Then $\delta(\mathcal N)$ exists and equals  $1-\Delta$. Hence by Theorem \ref{thm:density}, the relative logarithmic density of the Zhang primes is $1-\Delta$.
\end{theorem}

These results are completely analogous to Theorems 1.1 and 1.3 and Propositions 4.1 and 4.2 from~\cite{Lamz}.

Before moving on, we note a further consequence of the fact that $E_M(x)$ and $E_Z(x)$ possess the same explicit formula, namely that the symmetric difference of Mertens primes and Zhang primes has relative logarithmic density~$0$.
\begin{corollary}
Assume RH and LI. Then we have $\delta(S) = \delta^*(S) = 0$ for the symmetric difference 
$S=S_1\cup S_2$, where
$$
S_1=\{x:E_M(x)>0\ge E_Z(x)\}\quad\text{and}\quad S_2=\{x:E_Z(x)>0\ge E_M(x)\}.
$$
\end{corollary}
\begin{proof}
Take $\eta > 0$. 
Combining \cite[Corollary 2.2]{Lamz} with Corollary \ref{cor:Ex} and letting $T$ tend to infinity, we find that
\begin{align}\label{eq:symdiff}
    \big|E_M(x) - E_Z(x) \big|= O\Big(\frac{1}{\log x}\Big).
\end{align}
Let $c$ be the implied constant in equation~\eqref{eq:symdiff}. Thus for all $x \ge e^{c/\eta}$, if $E_M(x) > 0$ then $E_Z(x) > -\eta$. This means that
\begin{align*}
\delta^*(S_1) & \le \delta^*(\{x : E_Z(x) > 0\}) - \delta^*(\{x : E_Z(x) > -\eta\})\\
& = \delta(\{x : E_Z(x) > 0\}) - \delta(\{x : E_Z(x) > -\eta\}),
\end{align*}
which tends to 0 as $\eta \to 0$ by continuity, using Theorem \ref{thm:density} and Theorem \ref{thm:E_Z}. Since this holds for all $\eta >0$, we conclude that $\delta^*(S_1) = 0$. Interchanging the roles of $E_M$ and $E_Z$ proves $\delta^*(S_2) = 0$, and thus $\delta^*(S) = \delta^*(S_1\cup S_2)=0$. A similar argument (simpler even, without the appeal to Theorem \ref{thm:density}) shows that $\delta(S)=0$.
\end{proof}

We also remark, however, that the analogous argument does not work for $E_\pi$. This is because the relevant series over nontrivial zeros is $\sum_{\rho}x^{\rho-1}/(\rho-1)$ for $E_M$ and $E_Z$, while for $E_\pi$ it is $\sum_{\rho}x^\rho/\rho$. Assuming RH, this amounts to the observation that the two series $\sum_{\gamma}x^{i\gamma}/(-1/2+i\gamma)$
and $\sum_{\gamma}x^{i\gamma}/(1/2+i\gamma)$ are not readily comparable for a given $x$---even though, by symmetry, both do possess the same limiting distribution, which explains the appearance of $\delta(\Pi)=\Delta$ in results on the Mertens and Zhang races. 

The analogous problem of determining the density of the symmetric difference between the Mertens/Zhang primes and the $\li$-beats-$\pi$ primes is an interesting problem for further investigation; it would presumably proceed by examining the two-dimensional limiting distribution of the ordered pair of normalized error terms, and understanding how the correlations of the two functions' summands impacts the two-dimensional limiting distribution.

\section{Other series over prime numbers}
Before concluding our analysis, we remark that similar considerations apply more generally to series of the form $\sum_{p}p^\alpha (\log p)^{k+1}$, where $k\in\Z$ and $\alpha\in\R$. The basic approach is to first relate the sum of interest to the corresponding sum over prime powers via
\begin{align*}
    \sum_{p}\frac{\log p}{p^\alpha}(\log p)^k = \sum_{n} \frac{\Lambda(n)}{n^\alpha}(\log n)^k - \sum_{p^m, m\ge 2} \frac{m^k(\log p)^{k+1}}{p^{m\alpha}}.
\end{align*}
The next step is to employ an exact formula relating the sum over prime powers to series over zeros of $\zeta(s)$. For example, von Mangoldt's exact formula states that
\begin{align}
    \sum_{n\le x}\Lambda(n) = x - \frac{\zeta'}{\zeta}(0) - \sum_{\rho}\frac{x^{\rho}}{\rho} + \sum_{m\ge1} \frac{x^{-2m}}{-2m}
\end{align}
provided $x$ is not a prime power. The above formula naturally generalizes to any real exponent $\alpha$. Namely, one may prove by Perron's formula (c.f. \cite[Lemma 2.4]{Lamz}) that
\begin{align}
    \sum_{n\le x}\frac{\Lambda(n)}{n^\alpha} = \frac{x^{1-\alpha}}{1-\alpha} - \frac{\zeta'}{\zeta}(\alpha) - \sum_{\rho}\frac{x^{\rho-\alpha}}{\rho-\alpha}  + \sum_{m\ge1} \frac{x^{-2m-\alpha}}{-2m-\alpha}.
\end{align}
provided $x$ is not a prime power, and $\alpha$ is neither 1 nor a negative even integer. Note that when $\alpha>1$, we have $-(\zeta'/\zeta)(\alpha) = \sum_n \Lambda(n)/n^\alpha$ so we may simplify the above as
\begin{align}
    \sum_{n\ge  x}\frac{\Lambda(n)}{n^\alpha} = -\frac{x^{1-\alpha}}{1-\alpha} + \sum_{\rho}\frac{x^{\rho-\alpha}}{\rho-\alpha}  - \sum_{m\ge1} \frac{x^{-2m-\alpha}}{-2m-\alpha}.
\end{align}
To gain factors of $\log n$ in the numerator, we differentiate with respect to $\alpha$. Specifically, since $d/d\alpha[x^{c-\alpha}/(c-\alpha)] = (-\log x + 1/(c-\alpha))x^{c-\alpha}/(c-\alpha)$ for any $c\in\C$, by induction one  may show
$$\frac{d^k}{d\alpha^k}\Big[\frac{x^{c-\alpha}}{c-\alpha}\Big] = (-\log x)^k\bigg(\frac{x^{c-\alpha}}{c-\alpha} + O_k\Big(\frac1{\log x}\Big)\bigg)$$
This implies, for all $k\ge1$,
\begin{align*}
  &  \sum_{n\ge  x}\frac{\Lambda(n)}{n^\alpha}(\log n)^k \\
   &~ = (\log x)^k\bigg(-\frac{x^{1-\alpha}}{1-\alpha} + \sum_{\rho}\frac{x^{\rho-\alpha}}{\rho-\alpha}  - \sum_{m\ge1} \frac{x^{-2m-\alpha}}{-2m-\alpha}\bigg)\Big( 1 + O_k\Big(\frac{1}{\log x} \Big)\Big).
\end{align*}
Similarly for integration, we have 
$$\int_\alpha^\infty \frac{x^{c-\beta}}{c-\beta}\;d\beta = \li(x^{c-\alpha}) = (1 + O(1/\log x))\frac{x^{c-\beta}}{c-\beta},$$
so an induction argument will establish the exact formula for negative integers $k$.

From here, all that remains is to analyze the sum over nontrivial zeta zeros. Assuming RH, it suffices to consider the series $\sum_{\gamma}x^{i\gamma}/(1/2-\alpha+i\gamma)$. Further assuming LI, this series has a limiting distribution, which may be computed explicitly (as in \cite{ANS,RubSarn}).

\section*{Acknowledgments}
We are grateful for a helpful discussion with Dimitris Koukoulopoulos.
The first-named author thanks the office for undergraduate research at Dartmouth College, and is 
currently supported by a Churchill Scholarship at the University of Cambridge.
The second-named author was supported in part by a National Sciences and Engineering Research Council of Canada Discovery Grant.  The second- and third-named authors are grateful to the Centre de Recherches Math\'ematiques for
their hospitality in May, 2018 when some of the ideas in this paper were discussed.

\bibliographystyle{amsplain}

\end{document}